\documentclass[sn-mathphys-num]{sn-jnl}

\usepackage{graphicx}%
\usepackage{multirow}%
\usepackage{amsmath,amssymb,amsfonts}%
\usepackage{amsthm}%
\usepackage{mathrsfs}%
\usepackage[title]{appendix}%
\usepackage{xcolor}%
\usepackage{textcomp}%
\usepackage{manyfoot}%
\usepackage{booktabs}%
\usepackage{algorithm}%
\usepackage{algorithmicx}%
\usepackage{algpseudocode}%
\usepackage{listings}%

\newtheorem{theorem}{Theorem}[section]

\newtheorem{corollary}[theorem]{Corollary}

\newtheorem{lemma}[theorem]{Lemma}

\newtheorem{proposition}[theorem]{Proposition}

\begin{document}

\title{Characterization of Square Values and Power Sums of Consecutive Lucas Numbers}


\author*[]{\fnm{Priyabrata} \sur{Mandal}}\email{priyabrata@manit.ac.in}



\affil[]{\orgdiv{Department of Mathematics, Bioinformatics and Computer Applications}, \orgname{Maulana Azad National Institute of Technology Bhopal}, \orgaddress{ \city{Bhopal}, \postcode{462003}, \state{Madhya Pradesh}, \country{India}}}


\abstract{
	We establish several Diophantine results involving Lucas and Fibonacci numbers. First, we prove that $L_n=3x^2$ has the unique positive integer solution $(n,x)=(2,1)$. We also show that $F_n=5x^2$ admits only the solution $(n,x)=(5,1)$. We then prove that the equation $L_n^2+L_{n+1}^2=x^2$ has the unique solution $(n,x)=(2,5)$. Finally, we completely determine all non-negative integer solutions of the generalized equation $L_n^\alpha+L_{n+1}^\alpha=x^2.$
}




\maketitle

\section{Introduction}

Diophantine equations, named after the ancient mathematician Diophantus of Alexandria, are a class of equations where the solutions are restricted to integers. These equations are fundamental in number theory, with applications ranging from cryptography to the study of algebraic structures.
Their importance lies in the deep insights they provide into the nature of integers and the relationships between them, often revealing patterns and properties that are not immediately apparent.

The Lucas sequence, denoted by $(L_n)_{n \geq 0}$, is defined by the recurrence relation $L_{n+2} = L_{n+1} + L_n$ with initial conditions $L_0 = 2$ and $L_1 = 1$. This sequence, closely related to the Fibonacci numbers, has been extensively studied due to its rich mathematical properties and its role in various areas of number theory.

Recent research has focused on identifying instances where members of the Lucas sequence can be expressed as powers of integers. A significant contribution in this area is the work of Cohn \cite{cohn}, who established several key results concerning the Lucas and Fibonacci numbers, including conditions under which these numbers can be perfect squares or twice a perfect square. Cohn's results include:

\begin{enumerate}
    \item If $L_n = x^2$, then $n = 1$ or $n = 3$, i.e., $x = \pm 1$ or $\pm 2$.
    \item If $L_n = 2x^2$, then $n = 0$ or $n = \pm 6$, i.e., $x = \pm 1$ or $\pm 3$.
    \item If $F_n = x^2$, then $n = 0, \pm 1, 2$, or $12$, i.e., $x = 0$, $\pm 1$, or $\pm 12$.
    \item If $F_n = 2x^2$, then $n = 0, \pm 3$, or $6$, i.e., $x = 0$, $\pm 1$, or $\pm 2$.
\end{enumerate}
Despite significant progress in recent decades, complete classifications of such equations remain scarce and typically require refined techniques combining congruence arguments, quadratic residue theory, and properties of Lucas sequences. In particular, determining whether a given recurrence term or a simple expression involving consecutive terms can be a perfect square addresses fundamental questions about the interaction between linear recurrences and Diophantine geometry.

The present work contributes to this direction by establishing sharp uniqueness results for several natural Diophantine equations involving Fibonacci and Lucas numbers. The results demonstrate that these equations admit only trivial or exceptional solutions, thereby reinforcing the general phenomenon that perfect squares occur extremely rarely in structured recurrence sequences. Furthermore, the complete determination of solutions to the generalized equation 
$L_n^\alpha + L_{n+1}^\alpha = x^2$
 provides a unified framework that extends earlier isolated results and offers insight into higher-power analogues. Such classifications not only enrich the theory of Diophantine equations in recurrence sequences but also provide useful benchmarks for future investigations employing modular or transcendence methods.

\section{Preliminaries}
In this section, we present a series of theorems, lemmas, and commonly known identities related to generalized Fibonacci and Lucas numbers, which will be essential for proving the main results. Throughout this paper, the notation $\left(\frac{\ast}{\ast}\right)$ represents the Jacobi symbol.
\medskip 

The following is a well-known result concerning Fibonacci numbers. We give a proof of the lemma for completeness.
\begin{lemma}
Let $F_n$ denote the $n$-th Fibonacci number. Then, for all $n \geq 0$, we have $F_n^2 + F_{n+1}^2 = F_{2n+1}$.
\end{lemma}
\begin{proof}
    Let us prove the lemma using induction. For $n=0$, we have $F_0^2+F_1^2=0+1=1=F_1$. For $n=1$, $F_1^2+F_2^2=1+1=2=F_3$. Assume the induction hypothesis holds for all $k < n$. In particular, we have the following
    \begin{align}
        &F_{2n-3}=F_{n-2}^2+F_{n-1}^2 \quad (\text{ for }k=n-2) \label{F_2n-3 equation}\\
       \text{ and }
       &F_{2n-1}=F_{n-1}^2+F_{n}^2 \quad (\text{ for }k=n-1). \label{F_2n-1 equation}
    \end{align}
    For $k=n$, we have the following.
\begin{align*}
F_{2n+1} &= F_{2n} + F_{2n-1} \\
         &= 2F_{2n-1} + F_{2n-2} \\
         &= 3F_{2n-1} - F_{2n-3} \\
         &= 3(F_n^2 + F_{n-1}^2) - (F_{n-1}^2 + F_{n-2}^2)\  [\text{by } \text{ Equation }\ref{F_2n-3 equation} \text{ and } \ref{F_2n-1 equation}]\\
         &= 3F_n^2 + 2F_{n-1}^2 - F_{n-2}^2 \\
         &= 3F_n^2 + 2F_{n-1}^2 - (F_n - F_{n-1})^2 \\
         &= 2F_n^2 + 2F_nF_{n-1} + F_{n-1}^2 \\
         &= 2F_n^2 + 2F_n(F_{n+1} - F_n) + (F_{n+1} - F_n)^2 \\
         &= F_{n+1}^2 + F_n^2.
\end{align*}
\end{proof}
 In \cite{cohn}, Cohn proved the following properties concerning the Lucas numbers. We segregate these properties into the following theorem.
\begin{theorem}
    Let $L_n$ denote the $n$-th Lucas number. Then,
\begin{align}
&2 \mid L_n \text{ if and only if } 3 \nmid n. \label{2 divides L_n case}\\
 &\text{If } 3 \mid L_n, \text{then } n \equiv 2\ (mod\ 4). \label{3 divides L_n case}\\
    &L_n \equiv 3\ (mod\ 4) \text{ if } 2\mid n, 3\nmid n. \label{L_n is 4k+3 case}\\
    &\text{If } m \text{ is an integer such that } 2 \mid m, 3 \nmid m \text{ then } L_{n+2m} \equiv -L_n(mod\ L_m). \label{L_n+2m case} 
\end{align}
\end{theorem}
\medskip 

In 1844, Catalan \cite{catalan} proposed the conjecture that $(3, 2, 2, 3)$  is the only solution $(a, b, x, y)$ to the Diophantine equation $a^x - b^y = 1$ where $a$, $b$, $x$ and $y$ are integers such that $\min\{a, b, x, y\} > 1$. In 2004, Mihailescu \cite{mihailescu} proved Catalan's conjecture.

\begin{theorem} \label{catalan}
   Let \(a\), \(b\), \(x\), and \(y\) be integers such that \(\min\{a, b, x, y\}\) is greater than $1$. Then the Diophantine equation
\[
a^x - b^y = 1
\]
has a unique solution, given by \((a, b, x, y) = (3, 2, 2, 3)\).

\end{theorem}
\medskip 

The following result about Lucas numbers is established in \cite{cohn}.
\begin{theorem} \label{cohn theorem} 
The Diophantine equation $L_n=x^2$ has non-negative integer solutions if and only if $n=1$ or $n=3$.
\end{theorem}
\medskip 

The following lemma, stated in \cite{zhang}, is presented here with a proof.
\begin{lemma}\cite[Lemma 2.4]{zhang} \label{zhang}
Consider an odd prime $p$. Let $a$, $b$, $x$, and $n$ be integers such that $a,b$ are co-prime to each other and $n \geq 2$. If 
\[
a^p + b^p = x^n,
\]
then either $a + b = y^n$ or $a + b = p^{k-1}y^n$, for some integer $y$.
\end{lemma}

\begin{proof}
   We begin by expressing $\frac{a^p+ b^p}{a+b}$ using the formula
    \begin{equation*}
	\frac{a^p+ b^p}{a+b} = \sum_{i=0}^{p-1} (-1)^{i} a^{p-1-i} b^i=a^{p-1}-a^{p-1}\cdot b+a^{p-2}\cdot b^2- \cdots +b^{p-1}	
\end{equation*}
We claim that gcd$(a+b,\frac{a^p+ b^p}{a+b})=1$ or $p$.
Going modulo $(a+b)$, i.e., $b \equiv -a$, we get,
\begin{align*}
    a^{p-1}-a^{p-1}\cdot b+a^{p-2}\cdot b^2- \cdots +b^{p-1} &\equiv \underbrace{a^{p-1}+a^{p-1}+\cdots +a^{p-1}}_{p-\text{ many copies }}\ (mod\ a+b)\\
    &\equiv p \cdot a^{p-1}\ (mod\ a+b)
\end{align*}
Thus, we have gcd $(a+b,\frac{a^p+ b^p}{a+b})= \text{ gcd }(a+b, p \cdot a^{p-1})= \text{ gcd }(a+b,p)$, which proves our claim. Since $a^p+b^p=c^k$, it follows that $a+b=d^k$ or $a+b=p^{k-1}d^k$ depending on whether gcd $(a+b,\frac{a^p+ b^p}{a+b})$ is $1$ or $p$.
\end{proof}

Recall that, a solution $(a,b,x)$ to the equation $a^n+b^n=x^k$ is said to be {\it primitive} if gcd$(a,b,x)=1$.
The following theorem is established in \cite[Theorem 1]{poonen}.
\begin{theorem}[\cite{poonen}] \label{poonen}
The Diophantine equation $a^n + b^n = x^2$ does not admit any  nontrivial primitive solutions for $n \geq 4$.
\end{theorem}

\section{Main Results}
 In this section, we prove the main results of our paper. We start with the following result.
\begin{lemma}
Any two consecutive Lucas numbers are co-prime, i.e., 
\[\gcd(L_n, L_{n+1}) = 1 \text{ for all } n \geq 0\]
\end{lemma}
\begin{proof}
    Assume that $\gcd(L_n, L_{n+1}) = d$ for some positive integer $d$. Then, we have $d \mid L_n$ and $d \mid L_{n+1}$, meaning $d \mid (L_{n+1} - L_n) = L_{n-1}$. Similarly, we can show that $d \mid (L_n - L_{n-1}) = L_{n-2}$. Continuing this process, we deduce that $d \mid L_{n-3}, d \mid L_{n-4}, \dots, d \mid L_1 = 1$. Hence, $d = 1$.
\end{proof}

\begin{lemma}\label{Lucas of the form 3square}
    The only Lucas number $L_n$ that can be expressed in the form $L_n = 3x^2$ for some $x \in \mathbb{Z}$ is when $x = 1$, which occurs at $n = 2$.
\end{lemma}
\begin{proof}
Suppose $L_n = 3x^2$ for some $x \in \mathbb{Z}$. Since $3$ divides $L_n$, by Equation (\ref{3 divides L_n case}), we deduce that $n \equiv 2 \ (\text{mod} \ 4)$.
   First, we show that $L_n$ must be an odd number. Suppose, for the sake of contradiction, that $L_n$ is an even number. Then $2 \mid L_n$, and by Equation (\ref{2 divides L_n case}), it follows that $3 \nmid n$. Since $2 \mid n$ and $3 \nmid n$, by Equation (\ref{L_n is 4k+3 case}), we deduce that $L_n \equiv 3 \ (\text{mod} \ 4)$, which contradicts the assumption that $L_n$ is an even number. Therefore, $L_n$ must be an odd number.
 As $n \equiv 2 \ (\text{mod} \ 4)$, if $n = 2$, we have $L_2 = 3 = 3 \cdot 1^2$. Now, suppose $n \neq 2$.  Then $n$ can be expressed as $n = 2 + 2 \cdot 3^k \cdot m$, where $m$ is an integer such that $2 \mid m$ and $3 \nmid m$. Then, by repeated application of Equation (\ref{L_n+2m case}), we obtain
    \[
    L_n \equiv (-1)^{3^k} \cdot L_2\ (mod\ L_m) \equiv -3\ (mod\ L_m)
    \]
Thus, $3L_n \equiv -9 \ (\text{mod} \ L_m)$. Additionally, by Equation (\ref{L_n is 4k+3 case}), $L_m \equiv 3 \ (\text{mod} \ 4)$, and hence the Jacobi symbol 
    \[\left(\frac{3L_n}{L_m}\right)=\left(\frac{9}{L_m}\right)\left(\frac{-1}{L_m}\right)=-1.\]
    
    On the other hand, $3L_n = 3 \cdot 3x^2$, which is a perfect square, so the Jacobi symbol $\left(\frac{3L_n}{L_m}\right)=1$, leading to a contradiction. Thus, the only Lucas number $L_n$ that can be expressed in the form $L_n = 3x^2$ for some $x \in \mathbb{Z}$ is when $n=2$.
\end{proof}
\begin{lemma}\label{Fibonacci of the form 5x^2}
    The only Fibonacci number $F_n$ that can be expressed in the form $F_n=5x^2$ for some $x \in \mathbb Z$ is when $n=5,x=1$.
\end{lemma}
    \begin{proof}
        Recall that, $5$ divides $F_n$ if and only if $5$ divides $n$. Let $n=5k$. By \cite[Equation (23)]{robbins1} we have
        \[F_{5k}=5F_k(5F_k^4+5(-1)^k F_k^2+1).\]
        Since $F_n=F_{5k}=5x^2$, we have
        \begin{equation} \label{F_alpha case}
             x^2=F_k(5F_k^4+5(-1)^k F_k^2+1).
        \end{equation}
       
        Now $gcd\ (F_k,\  5F_k^4+5(-1)^k F_k^2+1)=1.$ Hence from the Equation (\ref{F_alpha case}), we deduce that $F_k=y^2$ and $5F_k^4+5(-1)^k F_k^2+1=z^2$ for some integers $y$ and $z$. As $F_k=y^2$, by \cite[Theorem 3]{cohn1}, we get $k = 1, 2, \text{ and } 12$. Consider the following cases for $k$. 
        \medskip 

        {\bf Case (i):} If $k=1$, then $n=5k=5$. Hence $F_5=5=5.1^2$ and so $F_n=5x^2$ has a solution for $n=5,x=1$.
        \medskip 

        {\bf Case (ii):} Suppose $k=2$, then $n=10$. In that case, $F_{10}=55=5\times 11$ which is not of the form $5x^2$.  
        \medskip 

        {\bf Case (iii):} If $k=12$, then $n=60$. By \cite[Theorem 3]{robbins}, $F_{60}=5x^2$ has no integer solutions. This completes the proof.
    \end{proof}

\begin{proposition}
The Diophantine equation $L_n^\alpha + L_{n+1}^\alpha = x^2$ with $n = 0$ has a unique solution in non-negative integers as $(n, \alpha, x)=(0, 3,3)$.
\end{proposition}
\begin{proof}
We have $L_0^\alpha + L_1^\alpha = x^2$. Substituting the initial values of the Lucas sequence, we get
\begin{align*}
 &2^\alpha + 1^\alpha = x^2\\
 \implies & x^2-2^\alpha=1.   
\end{align*}
 By Mihaileschu's Theorem \ref{catalan}, the only solution of the above in non-negative integer is $x=\alpha=3$. Hence, $(n,\alpha,x)=(0,3,3)$ is the only solution to the given Diophantine equation.
\end{proof}
\begin{theorem}\label{main theorem}
Let $L_n$ denote the $n$-th Lucas number. The complete set of non-negative solutions to the Diophantine equation $L_n^\alpha + L_{n+1}^\alpha = x^2$ for all $n > 0$ is given by $(n, \alpha, x) =\{(1,1,2), (2,2,5)\}$.
\end{theorem}

\begin{proof}

Let $n \neq 0$. We aim to find the integer solutions to the equation

\begin{equation}\label{main equation}
    L_n^\alpha + L_{n+1}^\alpha = x^2.
\end{equation}

We will now examine the following cases based on the value of $\alpha$.
\medskip 

\textbf{Case (i):} Let $\alpha = 0$. Substituting into Equation (\ref{main equation}), we obtain $1+1 = 2 = x^2$, which has no integer solutions for $x$.
\medskip 

\textbf{Case (ii):} Let $\alpha = 1$. In this case, Equation (\ref{main equation}) simplifies to
$L_n + L_{n+1} = x^2$. 
Since $L_n + L_{n+1}=L_{n+2}$, we get $L_{n+2} = x^2$. By Theorem \ref{cohn theorem}, the only solution is $x = 2$ and $n = 1$. Therefore, the only solution to Equation (\ref {main equation}) in this case is $(n, \alpha, x) = (1, 1, 2)$.
\medskip

 \textbf{Case (iii):} Consider the case when $\alpha=2$. From Equation (\ref{main equation}), we get $L_n^2 + L_{n+1}^2 = x^2$. Using the identity $L_n^2=4(-1)^n+5F_n^2$, where $F_n$ denote the $n$-th Fibonacci number, we get 
 \begin{align*}
     L_n^2+L_{n+1}^2 &=4(-1)^n+5F_n^2+4(-1)^{n+1}+5F_{n+1}^2\\
     &=5(F_n^2+F_{n+1}^2)\\
     &=5F_{2n+1}
 \end{align*}

Hence, we need to find integer solutions to the equation
\[
5F_{2n+1} = x^2.
\]
Since $5 \mid x^2 \implies 5 \mid x$, we can write $x = 5k$ for some integer $k$. Substituting this into the equation, we obtain
\[
F_{2n+1} = 5k^2.
\]
Thus, we now need to find integer solutions to the equation $F_{2n+1} = 5k^2$. By Lemma (\ref{Fibonacci of the form 5x^2}), we deduce that $k=1, 2n+1=5$, i.e., $x=5k=5$ and $n=2$. In this case, $(n, \alpha, x)=(2,2,5)$ is the only solution to the Equation (\ref{main equation}).
 
\medskip 

\textbf{Case (iv):} Let $\alpha = 3$. Then, Equation (\ref{main equation}) simplifies to
\[L_n^3 + L_{n+1}^3 = x^2.\] 
By applying Lemma \ref{zhang}, we obtain either $L_n + L_{n+1} = x^2$ or $L_n + L_{n+1} = 3x^2$. Note that the case $L_n + L_{n+1} = x^2$ leads to the same result as in Case (ii). On the other hand, by Lemma \ref{Lucas of the form 3square}, the equation $L_n + L_{n+1} = L_{n+2} = 3x^2$ implies $n = 0$, which results in a contradiction.
\medskip 

\textbf{Case (v):} Let $\alpha \geq 4$. Since any two consecutive Lucas numbers are co-prime, by Theorem \ref{poonen}, the Diophantine equation $L_n^\alpha + L_{n+1}^\alpha= x^2$ possesses no solutions in terms of non-negative integers. This completes the proof.
\end{proof}
\begin{corollary}
    There are no non-negative integer solutions to the Diophantine equation $L_n^\alpha + L_{n+1}^\alpha = x^k$ where $k=2d,$ with $ d>1$.
    \begin{proof}
       Suppose, to the contrary, that there exists a non-negative integer solution to $L_n^\alpha+L_{n+1}^\alpha=x^{2d}.$ Since $k=2d$, we can write the given equation as 
    $L_n^\alpha + L_{n+1}^\alpha = (x^d)^2$. By the above Theorem \ref{main theorem}, the only set of solutions are $(n,\alpha, x^d)=\{(1,1,2)), (2,2,5)\}$
    Consequently, we must have either $x^d=2$ or $x^d=5$. But neither equation $x^d=2$ nor $x^d=5$ admits an integer solution.
Therefore, the original Diophantine equation has no non-negative integer solutions.
    \end{proof}
\end{corollary}
\section*{Declarations}

\textbf{Funding:} The author(s) did not receive any financial support for the research, authorship, and/or publication of this article.
\medskip

\noindent \textbf{Competing interests:} The author(s) declare that they have no competing interests.









\end{document}